\documentclass[reqno,tbtags]{amsart}

\usepackage{amssymb,enumitem,cancel}
\usepackage{geometry}
\usepackage[foot]{amsaddr}
 
\usepackage{mathtools}
\usepackage{mathrsfs}
\usepackage{xfrac}
\usepackage{marginnote}
\usepackage{dsfont}
\usepackage{changepage}
\usepackage{xcolor}
\usepackage{soul}
\usepackage{extdash}
\usepackage{tikz-cd}
\usepackage{hyperref}
\usepackage{cleveref}
\numberwithin{equation}{section}

\newtheorem{thm}{Theorem}[section]
\newtheorem{prop}[thm]{Proposition}
\newtheorem{lem}[thm]{Lemma}
\newtheorem{cor}[thm]{Corollary}
\crefname{thm}{Theorem}{Theorems}
\crefname{prop}{Proposition}{Propositions}
\crefname{lem}{Lemma}{Lemmas}
\crefname{cor}{Corollary}{Corollaries}

\crefname{claim}{Claim}{Claims}
\crefname{step}{Step}{Steps}

\theoremstyle{definition}

\crefname{defn}{Definition}{Definitions}
\crefname{ex}{Example}{Examples}
\crefname{ass}{Assumption}{Assumptions}

\theoremstyle{remark}
\newtheorem{rmk}[thm]{Remark}
\crefname{rmk}{Remark}{Remarks}

\newtheorem*{theorem*}{Theorem}

\newtheoremstyle{namedass}%
    {}{}{}{}{\bfseries Assumption~}{.}{.5em}{#3}
\theoremstyle{namedass}
\newtheorem*{namedass}{Assumption}

\AddToHook{env/prop/begin}{\crefalias{thm}{prop}}
\AddToHook{env/lem/begin}{\crefalias{thm}{lem}}
\AddToHook{env/cor/begin}{\crefalias{thm}{cor}}
\AddToHook{env/claim/begin}{\crefalias{thm}{claim}}
\AddToHook{env/step/begin}{\crefalias{thm}{step}}
\AddToHook{env/defn/begin}{\crefalias{thm}{defn}}
\AddToHook{env/ex/begin}{\crefalias{thm}{ex}}
\AddToHook{env/rmk/begin}{\crefalias{thm}{rmk}}

\newcommand{\defeq}{\vcentcolon=}

\newcommand{\eqdef}{=\vcentcolon}

\renewcommand{\leq}{\leqslant}
\renewcommand{\geq}{\geqslant}

\let\temp\phi
\let\phi\varphi
\let\varphi\temp
\let\temp\varepsilon
\let\varepsilon\epsilon
\let\epsilon\temp

\DeclareMathOperator{\Lip}{Lip}

\DeclareMathOperator{\tr}{tr}

\newcommand{\N}{\mathbb{N}}
\newcommand{\Pc}{\mathcal{P}}

\newcommand{\de}{\partial}
\newcommand{\di}{\mathrm{d}}
\newcommand{\R}{{\mathbb{R}}}

\newcommand{\call}[1]{\ensuremath\mathcal{#1}}
\newcommand{\frk}[1]{\ensuremath\mathfrak{#1}}
\newcommand{\scr}[1]{\ensuremath\mathscr{#1}}
\newcommand{\bb}[1]{\ensuremath\mathbb{#1}}

\renewcommand{\bar}[1]{\ensuremath\overline{#1}}

\newcommand{\trn}{\mathsf{T}}
\newcommand\Sym{\scr S}

\newcommand\bs\boldsymbol

\newcommand\mres\llcorner

\makeatletter
\def\namedlabel#1#2{\begingroup
   \def\@currentlabel{#2}%
   \label{#1}\endgroup
}
\makeatother

\begin{document}


\title[Long-time and large-population limits of D-monotone LQ Nash systems]{Long-time and large-population limits of displacement monotone linear-quadratic Nash systems}
\author[D.\ F.\ Redaelli]{Davide Francesco Redaelli}
\address{Department of Mathematics\\ University of Rome Tor Vergata\\ Via della Ricerca Scientifica 1, 00133 Roma, Italy}
\email{redaelli@mat.uniroma2.it}
\thanks{The author is partially supported by the MUR Excellence Department Project Math@TOV, awarded to the Department of Mathematics of the University of Rome Tor Vergata, and by the Gruppo Nazionale per l’Analisi Matematica, la Probabilit\`a e le loro Applicazioni (GNAMPA) of the Istituto Nazionale di Alta Matematica (INdAM). He is also grateful to Marco Cirant for several useful discussions.}



\begin{abstract}
For linear-quadratic $N$-dimensional Nash systems with mean-field-like scaling and strongly displacement monotone data, we prove uniform-in-time a priori derivative estimates exhibiting the expected scaling as $N \to \infty$ in order to pass the system to the limit. As a consequence, in a heterogeneous (i.e., non-symmetric) mean-field setting, we obtain both uniform-in-time convergence of the Nash system on any arbitrarily long horizon $[0,T]$ (as $N \to \infty$), and convergence of both the $N$-dimensional Nash system and the corresponding master equation to their respective ergodic counterparts (as $T \to \infty$).
\end{abstract}

\maketitle


\section{Introduction}

We consider the following Nash system on $[0,T] \times (\R^d)^N$:
\begin{equation} \label{NS}
\begin{dcases}
-\de_t u^i - \Delta u^i + \frac12|D_i u^i|^2 + \sum_{\substack{1 \leq j \leq N \\ j \neq i}} D_j u^j \cdot D_j u^i = Q_{f^i \otimes I_d} 
\\[-5pt]
u^i(T,\cdot) = Q_{g^i \otimes I_d},
\end{dcases}
\qquad i \in \{1,\dots, N\},
\end{equation}
where $f^i, g^i \in \Sym(N)$ (i.e., symmetric $N\times N$ real matrices),  and $Q_A$ denotes the quadratic form associated to the matrix $\frac12 A \in \Sym(Nd)$, namely $Q_A(\bs x) \defeq \frac12 A\bs x \cdot \bs x$ for all $\bs x \in (\R^d)^N$. Our main assumptions on $f^i$ and $g^i$ are of \emph{mean-field-like scaling} and \emph{strong displacement monotonicity}; we refer to Assumptions~\ref{mfl_ass} and \ref{dmc_ass} in \Cref{sec:assmr} for the precise statements.

It is well-known that system~\eqref{NS} describes closed-loop Nash equilibria of the $N$-player linear-quadratic stochastic differential game on the horizon $[0,T]$ with states driven by the $\R^d$-valued SDEs
\[
\di X^i_t = \alpha_t^i \,\di t + \sqrt{2} \,\di W^i_t, \qquad i \in \{1,\dots, N\},
\]
for some independent $d$-dimensional Brownian motions $W^i$, and costs
\[
J^i(\bs\alpha) = \bb E\biggl[\, \int_0^T \Bigl( \frac12 |\alpha^i_t|^2 + Q_{f^i \otimes I_d}(\bs X_t) \Bigr)\,\di t + Q_{g^i \otimes I_d}(\bs X_T) \biggr];
\]
more precisely, closed-loop equilibrium controls of such a game are given by $-D_i u^i(t,\bs X_t)$, with $u^i$ solving \eqref{NS}.

Since Lions's lectures~\cite{LLL}, a central question in mean field game (MFG) theory has been the so-called \emph{convergence problem}, regarding the rigorous study of the large-population (i.e., $N\to\infty$) limit of (closed-loop) Nash equilibria and of the Nash system in the MFG setting. In this direction, the first major breakthrough was the book \cite{CDLL}, which solved the problem for a wide class of Nash systems enjoying the \emph{Lasry--Lions monotonicity} property, through a thorough study of smooth enough solutions to the \emph{master equation}---a PDE of hyperbolic nature on the space of measures, representing the expected limit for the Nash system.
Among the other contributions addressing the convergence problem, under various viewpoints, we recall \cite{CarLC}, \cite{CCP} for games with a major player, \cite{LackAAP,LackLF} for a more probabilistic approach, \cite{MZMem} expanding the study of the master equation, \cite{GMMZ} establishing the \emph{displacement monotone} regime as an alternative to the Lasry--Lions one, \cite{JJT} for interesting advances in the displacement monotone setting, and \cite{Dj} for the context of MFGs of controls (MFGC). 

Recently, a novel approach to the convergence problem in (heterogeneous) MFGs, with either a Lasry--Lions or a displacement monotone structure, have been proposed by Cirant and the author~\cite{CR-CPAM26}, based on a direct analysis of the $N$-dimensional Nash system, rather than on the study of some limit problem. Afterwards, the general strategy of \cite{CR-CPAM26} was improved in \cite{CJR25} to obtain non-asymptotic proximity estimates for different notions of equilibria for stochastic differential games on networks, and in \cite{JM25} to prove similar bounds in the setting of MFGC. 

Nevertheless, all the above works deal, in various forms, with the large-population limit of the Nash system, \emph{on a fixed time horizon $T$}, obtaining and making use of a priori estimates that are far from being uniform in $T$. The purpose of the present work is then to show that the approach of \cite{CR-CPAM26} can be adapted in order to obtain a priori bounds on the linear-quadratic Nash system which are stable with respect to both $T$ and $N$, thus allowing to address both the long-time and the large-population limits of \eqref{NS}. As the core strategy of \cite{CR-CPAM26} was first introduced in \cite{CR-DGA} in a linear-quadratic setting, we regard our results as a first step towards a more general approach to uniform-in-time non-asymptotic estimates for the Nash system.

In fact, the present work also provides, as far as we know, the first uniform-in-time convergence result for the Nash system. Previous papers dealing with long-time asymptotics in MFGs---and focusing on the MFG system---include \cite{Car13,CG15,CLLP1,CLLP2} in the Lasry--Lions monotone setting and \cite{CM25} in the displacement monotone setting (see also references therein), while different regimes are treated, e.g., by \cite{CCDE,CP21}. On the other hand, we mention the recent preprint \cite{CohZell} addressing the convergence problem in an ergodic framework. We point out that considering both the $N\to\infty$ and $T\to\infty$ limits, in general, seems to be challenging, especially due to the elusive nature of the ergodic Nash system---on which, to the best of our knowledge, very few existence results (see \cite{BenFre84,BenFreProc}) and no uniqueness results are available---and to the possibility for non-uniqueness phenomena to occur in the Lasry--Lions monotone scenario \cite{CarRai} which prevent a priori the attainability of uniform-in-time convergence estimates. 

Concerning the literature about linear-quadratic games, the convergence problem is studied with open-loop strategies, for instance, by \cite{HuaZ}, \cite{LMWZ} in a non-monotone setting with common noise, \cite{CecDia} in a non-uniqueness scenario establishing a selection principle, while \cite{CohJia} compares open-loop and closed-loop equilibria, and \cite{BarPri,PriDGA} deal with ergodic costs; long-time analysis, with open-loop strategies, is carried out for standard and graphon MFGs, with and without common noise, in the recent works \cite{BayJiaG,BayJia25,BayJia26}.

Finally, we point out that some of our structural assumptions may be relaxed to accomodate wider settings, with minor modifications in the proofs: it is possible to consider more general linear-quadratic games, such as
\[ \begin{gathered}
\di X^i_t = (A^iX^i_t - \alpha_t^i) \,\di t + \sigma^i \,\di W^i_t, 
\\
J^i(\bs\alpha) = \bb E\biggl[\, \int_0^T \Bigl( Q_{R^i}(\alpha^i_t) + Q_{f^i \otimes I_d}(\bs X_t-\bar{\bs x}_0) \Bigr)\,\di t + Q_{g^i \otimes I_d}(\bs X_T-\bar{\bs x}_0) \biggr],
\end{gathered}\]
for suitable matrices $A^i$, $\sigma^i$, $R^i > 0$, and arbitrary $\bar{\bs x}_0 \in (\R^d)^N$; also, the assumption of mean-field-like scaling can be relaxed, to include a wider class of games on sparser networks, in the same spirit as \cite{CJR25}. Such extensions will be addressed in an upcoming version of this work.

\section{Assumptions and outline of the main results} \label{sec:assmr}

We assume \emph{mean-field-like} interactions between players, and, to achieve uniform-in-time bounds, we consider \emph{strongly displacement monotone} costs; this is encoded in the next two main assumptions.

\begin{namedass}[(MF)] \namedlabel{mfl_ass}{\textbf{(MF)}}
For $\star \in \{f,g\}$, there exists a constant $ C_\star > 0$, independent of $N$, such that
\begin{equation} \label{eq_mfl_ass}
\sup_{1 \leq i \leq N} \biggl( \sum_{\substack{1 \leq k \leq N \\ k \neq i}} \Bigl( | \star^k_{ki} |^2 + | \star^i_{ik} |^2 \Bigr) + \frac{|\star^i_{ii} |^2}{N} \biggr) + \sum_{\substack{1 \leq i,h,k \leq N\\ h,k \neq i}} | \star^i_{hk} |^2 \leq \frac{ C_\star}{N}.
\end{equation}
\end{namedass}

\begin{rmk}
The bound \eqref{eq_mfl_ass} is a stronger version of \cite[Condition~(3.2)]{CR-DGA}. The terminology \emph{mean-field-like} is motivated by the fact that any $f$ such that
\begin{equation} \label{mflmot}
\sup_i |f^i_{ii}| + \sup_{j \neq i} N |f^i_{ij}| + \sup_{j \neq i \neq k} N^2|f^i_{jk}| \leq C,
\end{equation}
with $C$ independent of $N$, in fact satisfies \eqref{eq_mfl_ass}; and, in turn, \eqref{mflmot} is fulfilled (with $f^i_{jk} = D_{jk} F^i$) whenever $F^i(x) = \call F^i(x^i,\frac1{N-1} \sum_{j \neq i} \delta_{x^j})$ for any smooth enough function $\call F^i$ on $\R^d \times \call P(\R^d)$---see \cite[Proposition~6.1.1]{CDLL} or \cite[Remark~3.6]{CR-CPAM26}.
\end{rmk}

\begin{namedass}[(SD)] \namedlabel{dmc_ass}{\textbf{(SD)}}
For $\star \in \{f,g\}$, there exists a constant $M_\star > 0$, independent of $N$, such that
\begin{equation} \label{eq_dmcass}
B(\star) \bs x \cdot \bs x \geq M_\star |\bs x|^2 \qquad \forall\, \bs x \in \R^N,
\end{equation}
where $B(\star)_{ij} \defeq \star^i_{ij}$.
\end{namedass}

\begin{rmk}[Notation]
Even though the matrix $B(\star)$ is not symmetric in general, we will equivalently be writing $B(\star) \geq M_\star I_N$ as a compact form of \eqref{eq_dmcass}.
\end{rmk}

\begin{rmk}
Condition~\eqref{eq_dmcass} (with $\star = f$) is equivalent to
\[
\sum_{1 \leq i \leq N} \bigl( D_{x^i} Q_{f^i \otimes I_d}(\bs x) - D_{x^i} Q_{f^i \otimes I_d}(\bs y) \bigr) \cdot (x^i - y^i) \geq M_f |\bs x - \bs y|^2 \qquad \forall\, \bs x, \bs y \in (\R^d)^N;
\]
hence, following \cite[Remark~3.5]{CR-CPAM26}, it is the $N$-dimensional (linear-quadratic) counterpart of the following strong displacement monotonicity of $\call F \colon \R^d \times \call P_2(\R^d) \to \R$:
\[
\int_{\R^d \times \R^d} \bigl( D_x\call F(x,m_1) - D_x\call F(y,m_2) \bigr) \cdot (x - y) \,\mu(\di x,\di y) \geq M_f \int_{\R^d \times \R^d} |x-y|^2 \,\mu(\di x,\di y),
\]
for all $m_1,m_2 \in \call P_2(\R^d)$ and $\mu$ having marginals $m_1$ and $m_2$.
\end{rmk}

In order to exploit the linear-quadratic structure of the problem, we make the usual ansatz that the value function of the game has the form
\begin{equation} \label{LQ_ansatz}
u^i(t,\bs x) = Q_{c^i(t) \otimes I_d}(\bs x) + d \int_t^T \tr c^i(s) \,\di s, \qquad \bs x \in (\R^d)^N,
\end{equation} 
for some $c^i \colon [0,T] \to \Sym(N)$. Denoting by $e_i$ the $i$-th vector of the canonical basis of $\R^N$, straightforward computations show that the Nash system~\eqref{NS} can be recast into the following system of Riccati-type equations on $[0,T]$:
\begin{equation} \label{MF_sys}
 \begin{dcases}
- \dot c^i - (c^i e_i)^{\otimes 2} + B(c)^\trn c^i + c^i B(c) = f^i
\\
c^i(T) = g^i,
\end{dcases}
\qquad i \in \{1,\dots,N\}.
\end{equation}

Our first main theorem is the following.

\begin{thm} \label{mainthm}
Let Assumptions \ref{mfl_ass} and \ref{dmc_ass} be in force. Then there exist $N_0 \in \N$ and $C, M > 0$ (all depending only on $C_\star$ and $M_\star$, $\star \in \{f,g\}$) such that the following hold for all $N \geq N_0$.
\begin{enumerate}[leftmargin=*,label=(\alph*)]
\item For any $T > 0$, there exists a unique classical solution $\bs u^T$ of the form \eqref{LQ_ansatz} to the Nash system~\eqref{NS}, such that \eqref{eq_mfl_ass} and \eqref{eq_dmcass} holds for $\star = c^T$ with $(C_{c^T},M_{c^T}) = (C,M)$.
\item There exists $(v^i,\gamma^i)_{i = 1,\dots,N}$ and a constant $K > 0$ such that, for all $T>0$, $\eta \in (0,1)$, $i \in \{1,\dots,N\}$, $k \in \N$,\footnote{Note that as the functions considered are quadratic in space, the only relevant values of $k$ are actually $0$ and $1$.} 
and $(t,\bs x) \in [0,\eta T] \times (\R^d)^N$,
\begin{equation} \label{uiviconv}
\begin{split}
\frac{\bigl| u^{T,i}(t,\bs x) - ( v^i(\bs x) + (T-t) \gamma^i ) \bigr|^2}{1+|x^i|^4 + \frac1{N^2}|\bs x^{-i}|^{4}} 
&+ \frac{\bigl| D^k D_i \bigl(u^{T,i}(t,\bs x) - v^i(\bs x) \bigr) \bigr|^2}{1+|x^i|^{2(1-k)} + \frac1N |\bs x^{-i}|^{2(1-k)}} 
\\
&+ \sum_{\substack{1 \leq j \leq N \\ j \neq i}} \frac{N \bigl| D^k D_j \bigl(u^{T,i}(t,\bs x) - v^i(\bs x) \bigr) \bigr|^2}{1+|x^i|^{2(1-k)} + |\bs x^{-i}|^{2(1-k)}} 
\leq K e^{-(1-\eta)M T}.
\end{split}
\end{equation}
Moreover, the $N$-tuple of couples $(v^i,\gamma^i)_{i = 1,\dots,N}$ is the unique solution to the ergodic Nash system
\[
\gamma^i - \Delta v^i + \frac12|D_iv^i|^2 + \sum_{j \neq i} D_j v^j \cdot D_j v^i = Q_{f^i \otimes I_d}, \qquad i \in \{1,\dots,N\}
\]
with each $v^i$ quadratic and having Hessian $\frk c^i$ such that \eqref{eq_mfl_ass} and \eqref{eq_dmcass} hold for $\star = \frk c$ with $(C_{\frk c}, M_{\frk c}) = (C,M)$.
\end{enumerate}
\end{thm}

In fact, \Cref{mainthm} is obtained by combining the following two results, which reformulate points (a) and (b), respectively, at the level of the Riccati system \eqref{MF_sys}.

\begin{thm} \label{MF_thm}
Let Assumptions \ref{mfl_ass} and \ref{dmc_ass} be in force. There are $N_0 \in \N$ and $C,M >0$ such that, for all $N \geq N_0$ and for any $T > 0$, there exists a unique absolutely continuous solution $c^T$ to \eqref{MF_sys} on $[0,T]$, such that \eqref{eq_mfl_ass} and \eqref{eq_dmcass} hold for $\star = c^T$ with $(C_{c^T},M_{c^T}) = (C,M)$.
\end{thm}

This is an existence a uniqueness result for \eqref{MF_sys} on arbitrarily long horizons, also providing uniform-in-time derivative estimates with the correct scaling to pass to the $N\to\infty$ limit. The strategy is similar to that used to prove \cite[Theorem~3.1]{CR-DGA}, with the crucial difference that the positivity of $M_\star$ in Assumption~\ref{dmc_ass} is to be exploited to obtain long-time stability of the a priori bounds on $c = c^T$ (that is, on $D^2 \bs u = c \otimes I_d$).

In particular, we will follow the key idea in \cite{CR-DGA} of \emph{propagating a monotonicity bound} on $c$: we will start by \emph{assuming} to have a solution $c$ that is $M$-strongly displacement monotone (in the sense that \eqref{eq_dmcass} holds for $\star = c$ and $M_\star = M$) in a left neighbourhood of $T$, in order to obtain a priori bounds on $c$, stable with respect to $T$ and scaling according to \eqref{eq_mfl_ass}; then we will use those a priori estimates and the local solvability of \eqref{MF_sys} to infer that $c$ is in fact $2M$-strongly displacement monotone in the same neighbourhood of $T$, provided that $N$ is large enough (and that $M$ is suitably chosen with respect to the data), which eventually allows to infer that $c$ preserves the displacement monotone nature of $g$ from time $T$ up to $0$.

This time, we will need to simultaneously propagate another control on $c$ (see \eqref{extracond_loop}), as the a priori information on the displacement monotonicity of $c$ alone seems not to be enough to obtain all needed bounds. This also showcases one important flexibility of our method: roughly speaking, any number of the estimates one needs to prove can be put in this loop of improvement and propagations of bounds, as long as the circle eventually closes---that is, as long as the \emph{proved} a priori bounds in fact improve those \emph{assumed} estimates actually when $N$ is large enough.

Point (b) of \Cref{mainthm}, instead, is restated as follows.

\begin{thm} \label{convthm}
Let Assumptions \ref{mfl_ass} and \ref{dmc_ass} be in force, and, for any $T>0$, let $c^T$ be the solution from \Cref{MF_thm}. There exists $N_0' \in \N$ and $K>0$ (both independent of $T$) such that, if $N \geq N_0'$, then inequality \eqref{eq_mfl_ass} holds for $\star = c^T(t) - c^{T'\!}(t)$ with
\begin{equation} \label{convthmest}
\sup_{T' > T} \sup_{t \leq \eta T} C_{c^T(t)-c^{T'\!}(t)} \leq K e^{-(1-\eta)MT} \qquad \forall \eta \in (0,1).
\end{equation}
In particular, for $N \geq N_0'$, the net $\{c^T\}_{T>0}$ converges locally uniformly on $[0,+\infty)$  as $T \to +\infty$, to the unique solution $\frk c \in \Sym(N)^N$ to
\begin{equation} \label{stat_Ric}
 - (\frk c^i e_i)^{\otimes 2} + B(\frk c)^\trn \frk c^i + \frk c^i B(\frk c) = f^i, \qquad i \in \{1,\dots,N\}
\end{equation}
which satisfies \eqref{eq_mfl_ass} and \eqref{eq_dmcass} with $\star = \frk c$ and $(C_{\frk c},M_{\frk c}) = (C,M)$.
\end{thm}

Then, it is clear that, with Theorems~\ref{MF_thm} and \ref{convthm} at hand, it suffices to let
\begin{equation*} 
v^i(\bs x) \defeq Q_{\frk c^i \otimes I_d}(\bs x), \qquad \gamma^i \defeq d\tr \frk c^i = \tr(\frk c^i \otimes I_d) = \Delta v^i
\end{equation*}
to see that \Cref{mainthm} holds.

We note that the bound~\eqref{convthmest} is strictly related to exponential turnpike estimates implying long-time stabilisation of solutions to MFG systems; the reader can have a look at \cite{CM25}, and references therein. By a straightforward translation argument, under the assumptions of \Cref{MF_thm}, for all $N \geq N_0$ and for any $T>0$, the unique solution therein extends to the unique solution to \eqref{MF_sys} on $(-\infty,T]$, also preserving the constants $C$ and $M$. Therefore, the long-time convergence estimate \eqref{convthmest} for the net $\{ c^T \}_{T > 0}$ can be rephrased as a $t\to-\infty$  convergence estimate for a single such extended solution, instead; indeed, for any $T' > 0$,
\[
c^{T'\!}(t) = c^{T}(T-T'+t) \quad \forall \, t \leq T',
\]
so 
\begin{equation} \label{limlimcorr}
\sup_{T' > T} \sup_{t \leq \eta T} \bigl| c^T(t) - c^{T'\!}(t) \bigr| = \sup_{\tau>0} \sup_{t \leq \eta T} \bigl|c^T(t) - c^T(t-\tau)\bigr|
\end{equation}
for any $\eta \in (0,1)$. This is the point of view we are going to adopt to prove \Cref{convthm}, which then will follow from the next stability estimate.

\begin{prop} \label{prop_cctaudist}
Let $f(t),g,\bar g \in \Sym(N)^N$ satisfy Assumptions \ref{mfl_ass} and \ref{dmc_ass}, and let $c$ and $\bar c$ be the solutions to \eqref{MF_sys} on $(-\infty,T]$, with terminal conditions $g$ and $\bar g$, respectively. Then, for each $t \leq T$, inequality \eqref{eq_mfl_ass} holds for $\star = c(t) - \bar c(t)$ with
\[
C_{c(t)-\bar c(t)} = C_{g-\bar g} K e^{-M(T-t)},
\]
where $K$ depends only on $\max\{C_c, C_{\bar c}\}$ and $\min\{M_c, M_{\bar c}\}$.
\end{prop}

In \Cref{sec_lim}, as an application of \Cref{mainthm}, we address the problem of determining the large-population limit of both the evolutive and the ergodic Nash system; as a consequence, we also obtain that it in fact commutes with the long-time horizon limit (see \Cref{cor_comm} below). We are going to make the following additional structural assumption, replacing Assumption~\ref{mfl_ass}.

\begin{namedass}[(LMF)] \namedlabel{lim_ass}{\textbf{(LMF)}}
There exist parameters (\emph{labels}) $\lambda^i_N \in \Lambda \defeq [0,1]$ and a sequence of functions $\bar f_N \colon \Lambda \to \Sym(N)$ such that
for each $N \in \N$,
\begin{equation} \label{Qflim}
Q_{f^i \otimes I_d}(\bs x) = Q_{\bar f_N(\lambda^i_N) \otimes I_d}(x^i,\bs x^{-i}), \qquad \bs x \in (\R^d)^N,
\end{equation}
and likewise for $g^i$ (with corresponding function $\bar g_N$);
also, $Q_{\bar f_N \otimes I_d}, Q_{\bar g_N \otimes I_d} \colon \Lambda \times \R^d \times (\R^d)^{N-1} \to \R$ are Lipschitz continuous on $\Lambda$, uniformly in $N$, and invariant under permutations of the coordinates of $(\R^d)^{N-1}$. Furthermore, Assumption~\ref{mfl_ass} holds for $f^i$ defined by \eqref{Qflim} for any choice of $\lambda_N^i \in \Lambda$.
\end{namedass}

Such an assumption introduces a more concrete---and sufficiently regular---scenario of heterogeneous mean field games, akin to that considered in \cite[Section~7]{CR-CPAM26}, which allows to identify a limit value function (or, rather, a labeled family thereof) as $N\to\infty$, via standard compactness arguments.

In this regard, we have our second main result. We recall that $\call P_2(\call X)$ is the usual $2$-Wasserstein space of probability measures on $\call X$.

\begin{thm} \label{mainthmlim}
Let Assumptions~\ref{lim_ass} and \ref{dmc_ass} be in force. Let $\bs u^T = \bs u^T_N$ be the solution to \eqref{NS} and let $(\bs v,\bs\gamma) = (\bs v_N,\bs\gamma_N)$ be its corresponding ergodic limit, as in \Cref{mainthm}.
Then there exist unique functions
\[ \begin{gathered}
F, G \colon \Lambda \times \R^d \times \R^d \to \R,  \qquad U^T \colon \Lambda \times \Pc_2(\Lambda) \times [0,T] \times \R^d \times \R^d \to \R,
\\
V \colon \Lambda \times \Pc_2(\Lambda) \times \R^d \times \R^d \to \R, \qquad \Gamma \colon \Lambda \times \Pc_2(\Lambda) \to \R,
\end{gathered}
\]
Lipschitz continuous in $\Lambda \times \Pc_2(\Lambda)$, differentiable in $[0,T]$ and quadratic in $\R^d \times \R^d$, such that the following hold.
\begin{enumerate}[leftmargin=*,label=(\alph*)]
\item For all $i,k \in \N$, as $N \to \infty$, uniformly in $[0,T] \times K \times \Pc_2(\R^d)$ for any $K \subset \R^d$ compact, 
\[ \begin{gathered}
D_i^k u_N^{T,i}(t,\bs x) - D_x^k U^{T}(\lambda_N^i, m_{{\bs\lambda}_N^{-i}}; t, x^i, \beta(m_{{\bs x}^{-i}})) \to 0,
\\
N^k D_{j}^k u^{T,i}_N(t,\bs x) - D_b^k U^T(\lambda^i_N, m_{\bs\lambda^{-i}_N};t,x^i,\beta(m_{\bs x^{-i}})) \to 0 \qquad \forall\, j \neq i,
\end{gathered}
\]
where $\beta(m) \defeq \int_{\R} y \,\di m(y)$ denotes the barycentre of the measure $m$, and likewise for $\bs v_N$ and $V$ in place of $u^T_N$ and $U^T$; also
\[
\gamma_N^i - \Gamma(\lambda_N^i,m_{\bs\lambda_N^{-i}}) \to 0.
\]
\end{enumerate}
If $m_{\bs\lambda_N} \to \rho \in \Pc_2(\Lambda)$ as $N \to \infty$, then:
\begin{enumerate}[resume,leftmargin=*,label=(\alph*)]
\item $U^T$ solves the master equation
\begin{equation} \label{MEUT}
\begin{dcases} 
\begin{aligned}
&-\de_t U^T(\lambda,\rho;t,x,b) - \Delta_x U^T(\lambda,\rho;t) + \frac12 |D_x U^T(\lambda,\rho;t,x,b)|^2\\
&\quad + D_b U^T(\lambda,\rho;t,x,b) \cdot \int_{\Lambda} D_x U^T(\lambda',\rho;t,b,b) \,\di \rho(\lambda') = F(\lambda;x,b)
\end{aligned}
\\
U^T(\lambda,\rho;T,x,b) = G(\lambda;x,b),
\end{dcases}
\end{equation}
and
$(V,\Gamma)$ solves the ergodic master equation
\begin{equation} \label{MEV}
\begin{dcases} 
\begin{aligned}
&\Gamma(\lambda,\rho) - \Delta_x V(\lambda,\rho) + \frac12 |D_x V(\lambda,\rho;x,b)|^2 \\
&\quad + D_b V(\lambda,\rho;x,b) \cdot \int_{\Lambda} D_x V(\lambda',\rho;b,b) \,\di \rho(\lambda') = F(\lambda;x,b)
\end{aligned}
\\
V(\lambda,\rho;x,b) = G(\lambda;x,b);
\end{dcases}
\end{equation}
\item for all $T>0$, $\eta \in (0,1)$, $k \in \N$, $(\lambda,\rho,t,x,b) \in \Lambda \times \Pc_2(\Lambda) \times [0,\eta T] \times \R^d \times \R^d$,
\[
\frac{\bigl| (D_x)^k U^T(\lambda,\rho;t,x,b) - (D_x)^k \bigl( V(\lambda,\rho;x,b) + (T-t)\Gamma(\lambda,\rho)\bigr) \bigr|}{1 + |x|^2 + |b|^2} \leq \sqrt{K} e^{-\frac12(1-\eta)MT},
\]
where $K$ is the constant in \Cref{mainthm}(b).
\end{enumerate}
\end{thm}

We note that the assumption that $m_{\bs\lambda_N} \to \rho$ is more a definition of $\rho$ than a requirement for $\bs\lambda_N$, in the sense that by compactness we know that, up to subsequences, $m_{\bs\lambda_N}$ always converges in $\Pc_2(\Lambda)$. We also point out that, for more general data, one would expect the first-order evolutive master equation in this heterogeneous setting to be of the form
\[
\begin{dcases} 
\begin{aligned}
&-\de_t U - \Delta_x U - \int_{\Lambda \times \R^d} \Delta_y\frac{\delta U}{\delta \mu}(\lambda,t,x,\mu,\lambda',x') \,\di \mu(\lambda',x') + H(\lambda,x,D_xU)  \\
&+  \int_{\Lambda \times \R^d} \!\!\!D_pH(\lambda',x',D_xU(\lambda',t,x',\mu)) \cdot D_y \frac{\delta U}{\delta \mu} (\lambda,t,x,\mu,\lambda',x') \,\di \mu(\lambda',x')  = F(\lambda,x,(\pi_{\R^d})_\sharp \mu)
\end{aligned}
\\[7pt]
U(\lambda,T,x,\mu) = G(\lambda,x,(\pi_{\R^d})_\sharp \mu)
\end{dcases}
\]
for $U = U(\lambda,t,x,\mu)$ with $\mu \in \Pc_2(\Lambda \times \R^d)$; in fact, both \eqref{MEUT} and \eqref{MEV} are precisely the forms that this and its ergodic counterpart assume in the linear-quadratic framework.

Finally, as a direct consequence of point (b) of \Cref{mainthm} and points (a)--(c) of \Cref{mainthmlim}, we obtain the following suggestive result.

\begin{cor} \label{cor_comm}
Under the assumptions of \Cref{mainthmlim}, let $\bar \gamma_N \defeq \Delta_x \bar v_N$. Then, the following diagram commutes for each $k \in \N$, locally uniformly in $[0,T] \times \R^d \times \Pc_2(\R^d)$:
\[
  \begin{tikzcd}[every arrow/.style={draw,mapsto},labels={inner sep=1ex},column sep=1.5cm,row sep=1cm]
     (D_x)^k \bigl( \bar u_N^T - (T-t) \bar \gamma_N \bigr) \arrow[r,"N\to\infty"] \arrow[swap,d,"T \to \infty"] & (D_x)^k \bigl( U^T - (T-t)\Gamma \bigr) \arrow[d,"T \to \infty"] \\
     (D_x)^k \bar v_N \arrow[swap,r,"N\to\infty"] & (D_x)^k V \,.
  \end{tikzcd}
\]
\end{cor}

\section{A priori bounds on the Riccati system} \label{sec_bounds}

This section contains the needed estimates to prove Theorems~\ref{MF_thm} and \ref{convthm}, and the proofs of those results as well.

\subsection{Uniform-in-time scaling and monotonicity estimates}

As explained previously, the proof of \Cref{MF_thm} will be based on an argument of \emph{propagation} of some a priori bounds. To obtain the upcoming estimates, we will repeatedly use the following fact.

\begin{lem} \label{lem_psg}
Let $v \in \mathrm{AC}([0,T])$ solve
\[
\begin{dcases}
- \dot v \leq a - \mu v^\ell \quad \text{on}\ [0,T]
\\
v(T) = b,
\end{dcases}
\]
with $a,\mu,\ell > 0$ and $b \in \R$. Then
\[
v \leq \max \biggl\{ \Bigl( \frac{a}{\mu} \Bigr)^{\frac1\ell}, b \biggr\} \quad \text{on} \ [0,T].
\]
\end{lem}

\begin{proof}
It suffices to note that if $v$ must be strictly increasing on $X \defeq v^{-1}(((a/\mu)^{1/\ell},+\infty))$. Therefore, either $X = \emptyset$ or $\max_{[0,T]} v = v(T)$.
\end{proof}



\begin{prop} \label{prop_est2}
Let Assumption \ref{mfl_ass} be in force. Let $(c^i)_{1 \leq i \leq N}$ be an absolutely continuous solution to \eqref{MF_sys} on $(\tau,T]$ with $B(c) \geq MI_N$, for some $\tau \in [0,T)$ and $M > 0$ independent of $N$. Suppose also that 
\begin{equation} \label{extracond_loop}
\sum_{\substack{1 \leq i,h,k \leq N \\ h,k \neq i}} |c^i_{hk}|^2 \leq \frac{M^2}4 \quad \text{on} \ [\tau, T].
\end{equation}
Then there is a constant $K_M$ (depending only on $ C_g$, $ C_f$, and $M$) such that
\[
\sup_{1 \leq k \leq N} \sum_{i\neq k} \Bigl( |c^i_{ik}|^2 + |c^k_{ki}|^2  \Bigr) +  \sum_{\substack{1 \leq i,h,k \leq N \\ h,k \neq i}} |c^i_{hk}|^2 \leq \frac{K_M}N \quad \text{on} \ [\tau, T].
\]
\end{prop}

\begin{proof}
Multiplying the equation for $c^i_{ik}$ by $c^i_{ik}$ and summing over $i \neq k$ we obtain
\[  \begin{split}
-\frac12 \frac{\di}{\di t} \sum_{i\neq k} |c^i_{ik}|^2 &=  \sum_{i\neq k} f^i_{ik} c^i_{ik} - \sum_{i\neq k} c^k_{kk} |c^i_{ik}|^2 - \sum_{i,j \neq k} c^i_{ik} c^i_{ji}  c^j_{jk} - \sum_{\substack{i \neq k \\ j\neq i}} c^i_{ik} c^i_{jk}  c^j_{ji} \\
&\leq \frac{1}{4M} \sum_{i\neq k} |f^i_{ik}|^2 -  \frac{M}{2} \sup_k \sum_{i\neq k} |c^i_{ik}|^2,
\end{split} \]
so \Cref{lem_psg} gives
\[
\sup_k \sum_{i\neq k} |c^i_{ik}(t)|^2 \leq \frac1N \max\biggl\{ \frac{C_{f}}{2M^2}, C_g \biggr\} \eqdef \frac{C_0}{N}.
\]
On the other hand, summing over $k \neq i$ we obtain
\[  \begin{split}
-\frac12 \frac{\di}{\di t} \sum_{k\neq i} |c^i_{ik}|^2 &=  \sum_{k\neq i} f^i_{ik} c^i_{ik} - \sum_{k\neq i} c^i_{ii} |c^i_{ik}|^2 - \sum_{j,k \neq i} c^i_{ik} c^i_{ji}  c^j_{jk} - \sum_{j,k \neq i} c^i_{ik} c^i_{jk}  c^j_{ji} \\
&\leq \frac{1}{4M} \sum_{k\neq i} |f^i_{ik}|^2 + \frac1{4M} \sum_{j \neq i} |c^j_{ji}|^2 - \frac{3}{4M} \sum_{k\neq i} |c^i_{ik}|^2,
\end{split} 
\]
whence, by \Cref{lem_psg} and the previous estimate,
\[
\sup_i \sum_{k\neq i} |c^i_{ik}(t)|^2  
\leq \frac1N \max \biggl\{ \frac1{3M^2} \Bigl( \frac{C_f}{M} + C_0 \Bigr), C_g \biggr\} \eqdef \frac{C_1}{N}.
\]
We now multiply the equation for $c^i_{hk}$ by $c^i_{hk}$ and sum over $i$, $h \neq i$ and $k \neq i$, obtaining
\[ 
\begin{split}
- \frac12 \frac{\di}{\di t} \sum_{\substack{i,h,k \\ h,k\neq i}} |c^i_{hk}|^2 &=  \sum_{\substack{i,h,k \\ h,k\neq i}} f^i_{hk} c^i_{hk} - \sum_{\substack{i,j,h,k \\ j,h,k\neq i}} c^i_{hk} c^i_{jh}  c^j_{jk} - \sum_{\substack{i,j,h,k \\ h,k\neq i}} c^i_{hk} c^i_{jk}  c^j_{jh} 
\\
&\leq \frac{1}{2M} \sum_{\substack{i,h,k \\ h,k\neq i}} |f^i_{hk}|^2 + \frac1{2M} \sum_i \biggl(\, \sum_{h \neq i} |c^i_{ih}|^2 \biggr)^2 - M \sum_{\substack{i,h,k \\ h,k\neq i}} |c^i_{hk}|^2,
\end{split}
\]
so \Cref{lem_psg} again and the previous estimate give
\[
\sum_{\substack{i,h,k \\ h,k\neq i}} |c^i_{hk}(t)|^2 \leq \frac1N \max \Bigl\{ \frac1{2M^2}( C_f + C_1^2), C_g \Bigr\}. \qedhere
\]
\end{proof}

\begin{prop} \label{prop_est3}
Under the hypotheses of \Cref{prop_est2}, suppose further that $N \geq 4K_M/C_f^{\frac12}$, where $K_M$ is the constant given in \Cref{prop_est2}. Then there is a constant $K'>0$ (depending only on $C_g$ and $C_f$) such that
\begin{equation} \label{eq_est3}
 c^i_{ii} \leq K' \quad \text{on} \  [\tau, T].
\end{equation}
\end{prop}

\begin{proof}
Note that $c^i_{ii}$ solves
\[
- \dot c^i_{ii} + |c^i_{ii}|^2 + 2\sum_{j\neq i} c^j_{ji} c^i_{ij} = f^i_{ii},
\]
where by \Cref{prop_est2}
\[
\biggl| \sum_{j\neq i} c^j_{ji}(t) c^i_{ij}(t) \biggr| \leq \frac{K_M}{N}.
\]
Therefore,
\[
-\dot c^i_{ii} \leq C_f^{\frac12} - \frac{2K_M}{N} - |c^i_{ii}|^2 
\]
and by \Cref{lem_psg}
\[
c^i_{ii} \leq \max \biggl\{ \frac{C_f^{\frac14}}{\sqrt 2}, C_g^{\frac12} \biggr\}.
\qedhere
\]
\end{proof}

\begin{prop} \label{prop_prpg}
Under the hypotheses of \Cref{prop_est3}, suppose further that $N \geq 2K_M/M_f$.
Then
\[
B(c) \geq \min \biggl\{ M_g, \frac{M_f}{2K' + \sqrt{2M_f}} \biggr\} \quad \text{on} \ [\tau,T].
\]
\end{prop}

\begin{proof}
Note that $B(c)$ solves the equation
\begin{equation} \label{MF_eqB}
- \dot B(c) + B(c)^2 = B(f) - E, 
\qquad \text{with} \quad 
E_{ik} = \sum_{j\neq i} c^i_{kj} c^j_{ji}.
\end{equation}
By \Cref{prop_est2}, the Frobenius norm of $E$ is bounded by ${K_M}/N$.
Let now $\bs\xi$ solve
\[
\begin{dcases}
\dot {\bs\xi} = - B(c)^\trn \bs\xi \quad \text{on} \ [\tau,T]
\\
\bs\xi(\tau) = \bs\zeta \in \bb S^{N-1}.
\end{dcases}
\]
Note that by \Cref{prop_est2,prop_est3} we have
\[
B(c) \leq \bigl( K' + \tilde K \bigr) I_N,
\]
where we denoted by $K'$ the constant on the right-hand side of \eqref{eq_est3}, while $\tilde K \defeq \sqrt{K_M/N}$; therefore
\[
|\bs\xi(t)| \geq e^{-(K' + \tilde K)(t-\tau)} \quad \forall\, t \in [\tau,T].
\]
Using \eqref{MF_eqB} we get
\[ \begin{split}
- \bigl(B(c) \bs\xi \cdot \bs\xi \bigr)\dot{\phantom{|}} &= \bigl(B(f)-E + B(c)B(c)^\trn\bigr) \bs\xi \cdot \bs\xi 
\\
& \geq \bigl( M_f - \tilde K^2 \bigr) |\bs\xi|^2,
\end{split}
\]
whence
\[ \begin{split}
B(c)(\tau) \bs\zeta \cdot \bs\zeta &\geq B(g)\bs\xi(T) \cdot \bs\xi(T) + \bigl( M_f - \tilde K^2 \bigr) \int_\tau^T |\bs\xi(t)|^2 \,\di t
\\
&\geq M_g e^{-2(K' + \tilde K)(T-\tau)} + \frac{M_f - \tilde K^2}{K'+\tilde K} \Bigl( 1 - e^{-2(K' + \tilde K)(T-\tau)} \Bigr)
\\
&\geq \min \biggl\{ M_g, \frac{M_f - \tilde K^2}{K'+\tilde K} \biggr\},
\end{split}
\]
which yields the desired conclusion.
\end{proof}

We have now all the ingredients to prove \Cref{MF_thm}.

\begin{proof}[Proof of \Cref{MF_thm}]
Fix $M$ such that
\begin{equation} \label{boundonM}
0 < M < \min \biggl\{ M_g, \frac{M_f}{2K' + \sqrt{2M_f}} \biggr\}.
\end{equation}
By the Cauchy--Lipschitz theorem there exists $\tau \in [0,T)$ such that \eqref{MF_sys} has a unique absolutely continuous solution on $(\tau,T]$. Since $B(g) \geq M_g > M$, by taking $\tau$ closer to $T$ if necessary, by continuity we may suppose that $B(c) \geq M I$ on $(\tau,T]$. Furthermore, by Assumption \ref{mfl_ass} and continuity, provided that
\[
N > \frac{4C_g}{M^2},
\]
we may also suppose that condition~\eqref{extracond_loop} is satisfied. Therefore we can well define
\[
\tau^* \defeq \min \biggl\{ \tau \in [0,T) :\ 
\begin{array}{c}
\text{$\exists!$ $c \in \mathrm{AC}((\tau,T];\Sym(N))^N$ solving \eqref{MF_sys} on $(\tau, T]$}
\\
\text{with $B(c) \geq MI$ and satisfying \eqref{extracond_loop}}
\end{array}
  \biggr\}.
\]
We are going to prove that $\tau^* = 0$, provided that $N$ is chosen large enough; so suppose for a contradiction that $\tau^* > 0$. 
Let $c \in  \mathrm{AC}((\tau^*,T];\Sym(N))^N$ be the maximal solution satisfying the properties above. By \Cref{prop_est2,prop_est3}, $c$ continuously extends on $[\tau^*,T]$; therefore, by the Cauchy--Lipschitz theorem, $c$ extends to a solution to \eqref{MF_sys} on some interval $(\hat\tau,T]$ with $\hat\tau < \tau^*$. Without loss of generality, suppose that $\hat\tau$ is so close to $\tau^*$ that, by continuity, $B(c) \geq \frac12 MI$ on $[\hat\tau,T]$. By \Cref{prop_prpg} and \eqref{boundonM}, we have $B(c) \geq MI$ on $[\hat\tau,T]$, provided that
\[
N \geq \frac{2K_{\frac12M}}{\frac12C_f^{\frac12} \wedge M_f},
\]
while by \Cref{prop_est2}, we have that \eqref{extracond_loop} holds on $[\hat\tau,T]$, provided that
\[
N \geq \frac{16K_{\frac12 M}}{M^2}.
\]
This contradicts the minimality of $\tau^*$, thus proving that it must be $\tau^* = 0$.
\end{proof}

\subsection{Long-time stability of solutions} \label{sec_stab}

Now we prove \Cref{prop_cctaudist}, which in turn will allow to obtain \Cref{convthm}.

\begin{proof}[Proof of \Cref{prop_cctaudist}]
We are going to use the following notations, where $\Delta \defeq B(c) - B(\bar c) = B(c-\bar c)$:
\[ \begin{gathered}
\check\Delta_{jk} \defeq \Delta_{jk} \mathsf 1_{j \neq k}, \qquad \delta_i \defeq \Delta_{ii}, \qquad D^i_{jk} \defeq (c-\bar c)^i_{jk} \mathsf 1_{j \neq i \neq k}.
 \end{gathered}
\]
We have
\begin{equation} \label{est1_cD}
 \begin{split}
\frac12 \frac{\di}{\di t} |\check \Delta|^2 &= \sum_{i,k} \delta_k c^i_{ik} \check\Delta_{ik} + \sum_{i,k} \bar c^k_{kk} |\check \Delta_{ik}|^2 + \sum_{i,j,k} \check\Delta_{jk} c^{i}_{ij} \check\Delta_{ik} 
\\
&\quad + \sum_{\substack{i,j,k \\ j \neq k}} \Delta_{ij} \bar c^j_{jk} \check\Delta_{ik} + \sum_{i,j,k} c^i_{jk} \check\Delta_{ji} \check\Delta_{ik} + \sum_{i,j,k} D^i_{jk} \bar c^j_{ji} \check\Delta_{ik}
\\
&= \sum_{i,k} (\delta_k c^i_{ik} + \delta_i \bar c^i_{ik} ) \check\Delta_{ik}  + \sum_{i,j,k} \check\Delta_{jk} c^{i}_{ij} \check\Delta_{ik} 
\\
&\quad + \sum_{i,j,k} \check\Delta_{ij} \bar c^j_{jk} \check\Delta_{ik} 
+ \sum_{i,j,k} c^i_{jk} \check\Delta_{ji} \check\Delta_{ik} + \sum_{i,j,k} D^i_{jk} \bar c^j_{ji} \check\Delta_{ik}
\\
& \geq \biggl( M - \sqrt{\frac CN}\,\biggr) |\check\Delta|^2 - \frac1N \frac{C}{2M} \bigl( |\bs\delta|^2 + |D|^2 \bigr),
\end{split}
\end{equation}
where we have used the Cauchy--Schwarz and Young's inequalities, and the fact that both $c$ and $\bar c$ are $M$-strongly displacement monotone. 
Similarly,
\begin{equation} \label{est2_D}
\begin{split}
\frac12 \frac{\di}{\di t} |D|^2 &= 2 \sum_i \tr(D^i B(c) D^i) + 2 \sum_{\substack{i,j,h,k \\ j \neq i}} D^i_{hk} \bar c^i_{jh} \Delta_{jk} + \sum_{i,h,k} D^i_{hk} \Delta_{ik} c^i_{ih} + \sum_{i,h,k} D^i_{hk} \bar c^i_{ik} \Delta_{ih}
\\[-3pt]
&\geq M |D|^2 - \frac{4C}{NM} \bigl( |\check\Delta|^2 + |\bs\delta|^2 \bigr)
\end{split}
\end{equation}
and
\begin{equation} \label{est3_d}
\begin{split}
\frac12 \frac{\di}{\di t} |\bs\delta|^2 &= \sum_i \delta_i^2 (c^i_{ii} + \bar c^i_{ii}) + 2 \sum_{i,j} \delta_i \check \Delta_{ij} c^j_{ji} + 2 \sum_{i,j} \delta_i \bar c^i_{ij} \check \Delta_{ji}
\\
& \geq M |\bs\delta|^2 - \frac1N \frac{4C}{M} | \check \Delta |^2,
\end{split}
\end{equation}
Summing \eqref{est1_cD}, \eqref{est2_D} and \eqref{est3_d}, we obtain, for any $\epsilon > 0$,
\[
\frac{\di}{\di t} \bigl( |\check\Delta|^2 + |D|^2 + |\bs\delta|^2 \bigr) \geq (2-\epsilon)M \bigl( |\check\Delta|^2 + |D|^2 + |\bs\delta|^2 \bigr)
\]
provided that $N \geq \bar N(C,M,\epsilon)$ large enough. Note that, since  $|\check\Delta|^2 + |D|^2 + |\bs\delta|^2 = |c - \bar c|^2$, this is equivalent to
\[
\frac{\di}{\di t} |c -\bar c|^2 \geq (2-\epsilon) M |c - \bar c|^2
\]
and it implies that
\[
|c(t)-\bar c(t)|^2 \leq 3NC_{g-\bar g} e^{-(2-\epsilon)M(T-t)} \qquad \forall\, t \leq T.
\]

At this point we can go back to \eqref{est1_cD} and plug in such an estimate with $\epsilon \leq \frac14$, thus obtaining the following estimates for any $N \geq \bar N'(C,M,\epsilon)$ large enough. First,
\begin{equation} \label{finest_cD}
|\check\Delta(t)|^2 \leq C_{g-\bar g} e^{-(2-2\epsilon)M(T-t)} + \frac{3C^2}{2NM}(T-t)e^{-(2-\epsilon)M(T-t)} \leq 2C_{g-\bar g} e^{-(2-2\epsilon)M(T-t)} \quad \forall t \leq T.
\end{equation}
Similarly, as we can also estimate
\[
\frac12 \frac{\di}{\di t} |\delta_i|^2 \geq M|\delta_i|^2 - \frac1N \frac{4C}{M} |\check\Delta|^2 \qquad \forall \, i \in \{1,\dots,N\},
\]
we obtain
\begin{equation} \label{finest_d}
\sup_{1 \leq i \leq N} |\delta_i(t)|^2 \leq 2C_{g-\bar g} e^{-(2-2\epsilon)M(T-t)} \qquad \forall t \leq T.
\end{equation}
Then note that we have the following estimate, alternative of \eqref{est2_D}:
\[
\frac12 \frac{\di}{\di t} |D|^2 \geq M|D|^2 - \frac{6C}{NM} \Bigl(|\check\Delta|^2 + \sup_i |\delta_i|^2 \Bigr);
\]
therefore, using \eqref{finest_cD} and \eqref{finest_d},
\[
|D(t)|^2 \lesssim \frac{e^{-(2-3\epsilon)M(T-t)}}{N},
\]
with implied constant (also below) depending only on $C$, $M$ and $C_{g-\bar g}$.
Finally, we can start from the evolutions of $\sum_{k \neq i} |c^i_{ik}-\bar c^i_{ik}|^2$ and $\sum_{k \neq i} |c^k_{ki} - \bar c^k_{ki}|^2$ and then use the above bounds in a similar way to obtain
\[
\sum_{k \neq i} \bigl( |c^i_{ik}(t) - \bar c^i_{ik}(t)|^2 + |c^k_{ki}(t) - \bar c^k_{ki}(t)|^2 \bigr) \lesssim \frac{e^{-(2-4\epsilon)M(T-t)}}{N}.
\]
We omit the details of these last estimates.
\end{proof}

As said, this is sufficient to prove \Cref{convthm}.

\begin{proof}[Proof of \Cref{convthm}]
The first part comes from identity \eqref{limlimcorr} and \Cref{prop_cctaudist} when $\bar g = c(T-\tau)$, so that $\bar c = c(\cdot - \tau)$. We only need to prove the last claim. Recall that the convergence of $\{c^T\}_{T>0}$ is equivalent to that of $\{c^{T_0}(t)\}_{t \leq T_0}$ as $t \to - \infty$ for any fixed $T_0>0$. So, as $\lim_{t \to -\infty} c^{T_0}(t)$ is clearly independent of $T_0$, we deduce that $\frk c(t)$ is independent of $t$, i.e.\ it is constant. By the equation of $c = c^{T_0}$ and the fundamental theorem of calculus, $\lim_{t \to - \infty} \dot c(t) = 0$, then $\frk c$ solves \eqref{stat_Ric} by passing \eqref{MF_sys} to the limit as $t \to - \infty$. To prove that $\frk c$ is the unique solution to \eqref{stat_Ric} satisfying \eqref{eq_mfl_ass} and \eqref{eq_dmcass} with $\star = \frk c$ and $(C_{\frk c},M_{\frk c}) = (C,M)$, note that any other such solution $\bar{\frk c}$ would also be a stationary solution to \eqref{MF_sys} with terminal condition $\bar g = \bar{\frk c}$. Then by \Cref{prop_cctaudist}, letting $t \to -\infty$ we deduce that $\frk c = \bar{\frk c}$.
\end{proof}

\section{Large-population limits} \label{sec_lim}

We recall that in this section Assumption~\ref{lim_ass} replaces \ref{mfl_ass} as a strengthen version of it; Assumption~\ref{dmc_ass} is still in force as is.
We are going to approach the proof of \Cref{mainthmlim} with a sequence of remarks addressing some implications of this setting. In the following, in order to ease the notation, we will set (without loss of generality) $d=1$.

\begin{rmk}[Existence of a representative value function]  \label{rmk_cTsym}
By symmetry, for each $T > 0$ there exists $\bar c^T_N \colon \Lambda \times \Lambda^{N-1} \times [0,T] \to \Sym(N)$ such that the solution to the $N$-dimensional linear-quadratic Nash system is given by
\[ \begin{split}
u^{T,i}(t,\bs x) &= \bar u^T_N(\lambda^i_N,\bs\lambda_N^{-i};t,x^i,\bs x^{-i}) 
\\
&\defeq Q_{\bar c^T_N(\lambda_N^i,\bs\lambda_N^{-i};t)}(x^i,\bs x^{-i}) + \int_t^T \tr \bar c^T_N(\lambda_N^i, \bs\lambda_N^{-i};s)\,\di s, \qquad \bs x \in \R^N
\end{split}
\]
and is invariant under permutations of the coordinates of $\Lambda^{N-1}$ and $\R^{N-1}$ (i.e.\ the second and fifth variables of $\bar u^T_N$), and likewise for the infinite horizon limit $v^i$, with corresponding functions $\bar v_N$ and $\bar{\frk c}_N$.
\end{rmk}

\begin{rmk}[The quadratic forms as functions on probability measures] \label{rmk_qfp}
The symmetry assumption in the space variables is equivalent to asking that
\[
\bar f_N = \begin{pmatrix}
(\bar f_N)_{11} & (\bar f_N)_{12} \bs 1_{N-1}^\trn \\
(\bar f_N)_{12} \bs 1_{N-1} & (\bar f_N)_{22} I_{N-1} +(\bar f_N)_{23} \bigl( \bs 1_{N-1}^{\otimes 2} - I_{N-1} \bigr)
\end{pmatrix}
\]
where $\bs 1_{N-1}$ is the $(N-1)$-dimensional vector of ones, and likewise for $\bar g_N$, $\bar c^T_N$ and $\bar{\frk c}_N$. Then
\[ \begin{split}
Q_{\bar f_N}(x^i, \bs x^{-i}) &= \frac12 (\bar f_N)_{11} |x^i|^2 + (N-1)(\bar f_N)_{12} x^i \cdot \Bigl( \frac1{N-1} \sum_{j \neq i} x^j \Bigr) 
\\
&\quad + \frac12 (N-1)\bigl( (\bar f_N)_{22} - (\bar f_N)_{23}\bigr) \frac1{N-1} \sum_{j \neq i} |x^j|^2 + (N-1)^2(\bar f_N)_{23}\Bigl| \frac1{N-1} \sum_{j \neq i} x^j \Bigr|^2 ,
\end{split}
\]
and upon defining, for $x \in \R$ and $m \in \Pc_2(\R)$,
\[ \begin{split}
F_N(\lambda; x,m) &\defeq \frac12 \bar f_N(\lambda)_{11} |x|^2 + (N-1)\bar f_N(\lambda)_{12} x \cdot \beta(m)
\\
&\quad + \frac12(N-1)\bigl( \bar f_N(\lambda)_{22} - \bar f_N(\lambda)_{23}\bigr) \int_\R |y|^2 \,\di m(y) + (N-1)^2\bar f_N(\lambda)_{23} |\beta(m)|^2
\end{split}
\]
(where we recall that $\beta(m) \defeq \int_{\R} y \,\di m(y)$), 
we see that
\[
Q_{\bar f^i(\lambda)}(\bs x) = F_N(\lambda; x^i,m_{\bs x^{-i}}).
\]
Analogous considerations also hold for $\bar g_N$ and, in view of \Cref{rmk_cTsym}, for $\bar c^T_N$ and $\bar{\frk c}_N$. 
\end{rmk}

\begin{rmk}[Large-population limit data] \label{rmk_lplf}
By \Cref{rmk_qfp}, estimate \eqref{eq_mfl_ass} of Assumption~\ref{mfl_ass} yields
\begin{equation} \label{estdec}
|\star_{11} | \lesssim 1,
\qquad
| \star_{12} | \lesssim \frac1N,
\qquad
| \star_{22} | \lesssim \frac1{N^{\frac32}},
\qquad
| \star_{23} | \lesssim \frac1{N^2},
\end{equation}
for $\star \in \{\bar f_N, \bar g_N \}$. This implies that there exists $\call F \colon \Lambda \to \Sym(2)$ such that 
we have, up to subsequences, 
\[
(D_x)^k Q_{\bar f_N(\lambda)}(x^i,\bs x^{-i}) - (D_x)^k Q_{\call F(\lambda)}(x^i, \beta(m_{\bs x^{-i}})) \to 0 \qquad \text{as} \ \ N \to \infty \qquad \text{for} \ k \in \N
\]
uniformly in $\Lambda \times K \times \Pc_2(\R)$ for any $K \subset \R$ compact. Furthermore, one actually sees that such an $\call F$ must be unique, so the above convergence holds along the whole sequence. An analogous convergence holds for $Q_{\bar g_N}$.
\end{rmk}

\begin{rmk}[Displacement monotonicity revisited] \label{rmk_dmr}
According to \Cref{rmk_qfp}, we have
\[
B(f) = \mathrm{diag}\bigl((\bar f_N)_{11}(\lambda^i_N)\bigr)_{1 \leq i \leq N} + \mathrm{diag}\bigl((\bar f_N)_{12}(\lambda^i_N)\bigr)_{1 \leq i \leq N} \bigl( \bs 1_{N}^{\otimes 2} - I_{N} \bigr),
\]
so
\[ \begin{split}
B(f)\bs x \cdot \bs x &= \sum_j  (\bar f_N)_{11}(\lambda^j_N) |x^j|^2 + 2\sum_{\substack{i,j \\ i \neq j}} (\bar f_N)_{12}(\lambda^i_N) x^i \cdot x^j
\\
&= N \int_{\Lambda \times \R} \Bigl( \bigl( (\bar f_N)_{11}(\lambda) - 2(\bar f_N)_{12}(\lambda) \bigr) |x|^2 + 2N \beta(m_{\bs x}) \cdot (\bar f_N)_{12}(\lambda) x \Bigr) \,\di m_{(\bs\lambda_N:\bs x)}(\lambda,x),
\end{split}
\]
with $m_{(\bs\lambda_N:\bs x)} \defeq \frac1N \sum_j \delta_{(\lambda^j_N,x^j)} =  \frac1N \sum_j \delta_{\lambda^j_N} \otimes \delta_{x^j}$. By compactness there exists $\rho \in \Pc_2(\Lambda)$  such that, up to subsequences, $m_{\bs\lambda_N} \to \rho$ as $N \to \infty$, and, given $m \in \Pc_2(\R)$, we can choose $\bs x \in \R^N$ in such a way that $m_{\bs x} \to m$. 
Then, in the large-population limit, Assumption~\ref{dmc_ass} yields
\begin{equation} \label{pardismon}
\int_{\Lambda \times \R \times \R} \Bigl( D_x Q_{\call F(\lambda)}(y,\beta(m)) - D_x Q_{\call F(\lambda)}(y',\beta(m')) \Bigr) \cdot (y-y') \,\di \varpi(\lambda,y,y') \geq M_f W_2(m,m')^2,
\end{equation}
for all $m,m' \in \Pc_2(\R)$ and {some} $\varpi \in \Pc_2(\Lambda \times \R \times \R)$ that has $\rho$, $m$ and $m'$ as its marginals. Note that in fact \eqref{pardismon} is a natural generalisation of a strong displacement monotonicity condition to the label-dependent function $(\lambda,x,m) \mapsto Q_{\call F(\lambda)}(x,\beta(m))$.
Conversely, it is also possible to show that we could modify Assumption~\ref{lim_ass} by define $Q_{f^i}$ by projecting onto the space of empirical measures a function $Q_{\call F(\lambda)}$ satisfying \eqref{pardismon} (and likewise for $Q_{g^i}$), and have Assumptions~\ref{mfl_ass} and \ref{dmc_ass} verified.
\end{rmk}

We are now ready to prove our second main theorem.

\begin{proof}[Proof of \Cref{mainthmlim}]
From \Cref{rmk_cTsym} and \Cref{mainthm}, we know that \eqref{estdec}
also holds for $\star \in \{\bar c^T_N, \bar{\frk c}_N \}$.  On the other hand, the Lipschitz regularity with respect to $\bs\lambda \in \Lambda^N$ of the data is also transferred to the solutions $\bar c^T_N$ and $\bar{\frk c}_N$, in a suitably weighted manner. In fact, from \Cref{prop_cctaudist} when $g^i = \bar g(\lambda^i)$ and $\bar g^i = \bar g(\bar \lambda^i)$, by precisely tracking the constants appearing in \eqref{finest_cD}, \eqref{finest_d} and the following estimates, we see that, for all $\lambda,\lambda' \in \Lambda$, $\bs\lambda, \bs{\lambda}' \in \Lambda^{N-1}$,
\[
\sup_{T>0} \sup_{t \in [0,T]} \sum_{1 \leq j,k \leq N} \bigl| (\bar c^{T}_N)_{jk}(\lambda, \bs\lambda, t) - (\bar c^{T}_N)_{jk}(\lambda', \bs{\lambda}',t) \bigr|^2 \leq \bar C \Bigl( |\lambda - \lambda'|^2 + \frac{|\bs\lambda - \bs{\lambda}'|^2}{N} \Bigr),
\]
with $\bar C$ depending only on $C$, $M$ and the Lipschitz constant of $\bar g_N$---and thus, in particular, independent of $N$; then, letting $T \to +\infty$, we also have
\[
\sum_{1 \leq j,k \leq N} \bigl| (\bar{\frk c}_N)_{jk}(\lambda, \bs\lambda) - (\bar{\frk c}_N)_{jk}(\lambda', \bs{\lambda}') \bigr|^2 \leq \bar C \Bigl( |\lambda - \lambda'|^2 + \frac{|\bs\lambda - \bs{\lambda}'|^2}{N} \Bigr).
\]
In particular, the above estimates with $\lambda = \lambda'$ imply that the maps $\bs\lambda \mapsto \bar c^T_N(\lambda,\bs\lambda)$ and $\bs\lambda \mapsto \bar {\frk c}_N(\lambda,\bs\lambda)$ are Lipschitz continuous from $\Lambda^{N-1}$---equipped with the distance Wasserstein distance $W_2$, induced by embedding $\Lambda^{N-1} \hookrightarrow \Pc_2(\Lambda)$ via $\bs\lambda \mapsto m_{\bs\lambda} \defeq \frac1{N-1} \sum_j \delta_{\lambda^j}$---to $\Sym(N)$ with the Frobenius norm.

By analogous arguments as for \cite[Theorem~7.5]{CR-CPAM26}, estimates~\eqref{estdec} together with the Lipschitz continuity above ensures the existence of a function $\call U^T\in \Lip( \Lambda \times \Pc_2(\Lambda)\times [0,T];\Sym(2))$ such that, given
\[
\begin{split}
U^T(\lambda,\rho;t,x,b) &\defeq Q_{\call U^T(\lambda,\rho;t)}(x,b)
+ \int_t^T \call U^T(\lambda,\rho;s)_{11} \,\di s,
\end{split}
\]
we have, up to subsequences, 
\[
(D_x)^k \bar u_N^T(\lambda,\hat{\bs\lambda};t,x,\hat{\bs x}) - (D_x)^k U^{T}(\lambda, m_{\hat{\bs\lambda}}; t, x, \beta(m_{\hat{\bs x}})) \to 0 \qquad \text{as} \ \ N \to \infty \qquad \text{for} \ k \in \N
\]
uniformly in $\Lambda \times \Pc_2(\Lambda) \times [0,T] \times K \times \Pc_2(\R)$ for any $K \subset \R$ compact. Like $\call F$ in \Cref{rmk_lplf}, $\call U^T$ is in fact unique, so the above convergence holds for the whole sequence. Similar arguments hold for $\bar v_N$, so we identify in this way a limit ergodic value function $V$ as well.

At this point, by computing the derivatives of $U^T$, it is easy to see that the following convergence also holds, as $N\to\infty$, uniformly in $\Lambda \times \Pc_2(\Lambda) \times [0,T] \times K \times \Pc_2(\R)$:
\[
N^k (D_{x^j})^k \bar u^T_N(\lambda^i,\bs\lambda^{-i};t,x^i,\bs x^{-i}) - (D_b)^k U^T(\lambda^i, m_{\bs\lambda^{-i}};t,x^i,\beta(m_{\bs x^{-i}})) \to 0 \qquad \text{for $k \in \N$, $j \neq i$.}
\]
This implies that
\[
\Delta u^T_N(\lambda^i,\bs\lambda^{-i};t,x^i,\bs x^{-i}) - \Delta_{x^i} U^T(\lambda^i, m_{\bs\lambda^{-i}};t,x^i,\beta(m_{\bs x^{-i}})) \to 0
\]
as well as
\[ \begin{split}
&\sum_{\substack{1 \leq j \leq N \\ j \neq i}} D_{x^j} \bar u^T_N(\lambda^j,\bs\lambda^{-j};t,x^j,\bs x^{-j}) D_{x^j} \bar u^T_N(\lambda^i,\bs\lambda^{-i};t,x^i,\bs x^{-i})
\\
&\qquad - \int_{\Lambda \times \R} D_x U^T(\lambda,m_{\bs\lambda^{-i}};t,x,m_{\bs x^{-i}}) D_b U^T(\lambda^i, m_{\bs\lambda^{-i}};t,x^i,\beta(m_{\bs x^{-i}})) \,\di( m_{\bs\lambda^{-i}} \otimes m_{\bs x^{-i}} )(\lambda,x)
\\[5pt]
&= \int_{\Lambda \times \R} \Bigl( N D_{x} \bar u^T_N(\lambda,\bs\lambda^{-i};t,x,\bs x^{-i}) D_{x^j} \bar u^T_N(\lambda^i,\bs\lambda^{-i};t,x^i,\bs x^{-i})
\\
&\qquad - D_x U^T(\lambda,m_{\bs\lambda^{-i}};t,x,m_{\bs x^{-i}}) D_b U^T(\lambda^i, m_{\bs\lambda^{-i}};t,x^i,\beta(m_{\bs x^{-i}})) \Bigr) \di( m_{\bs\lambda^{-i}} \otimes m_{\bs x^{-i}} )(\lambda,x)
\\
&\qquad + \frac1N \sum_{\substack{1 \leq j \leq N \\ j \neq i}} D_x \bar u^T_N\Bigr|^{(\lambda^j,\bs\lambda^{-j};t,x^j,\bs x^{-j})}_{(\lambda^j,\bs\lambda^{-i};t,x^j,\bs x^{-i})} D_b U^T(\lambda^i, m_{\bs\lambda^{-i}};t,x^i,\beta(m_{\bs x^{-i}})) \ \longrightarrow \ 0,
\end{split}
\]
where we also used that
\[
\biggl| D_x \bar u^T_N\Bigr|^{(\lambda^j,\bs\lambda^{-j};t,x^j,\bs x^{-j})}_{(\lambda^j,\bs\lambda^{-i};t,x^j,\bs x^{-i})} \biggr| \lesssim W_2(m_{\bs\lambda^{-j}},m_{\bs\lambda^{-i}}) + W_2(m_{\bs x^{-j}},m_{\bs x^{-i}}) \lesssim \frac{|\lambda^i - \lambda^j| + |x^i - x^j|}{\sqrt N},
\]
which comes from the Lipschitz continuity with respect to $\bs\lambda$ and standard properties of the Wasserstein distance. Analogous considerations are also true when considering $\bar v_N$ instead of $\bar u^T_N$.

This proves the existence and the structure of $U^T$ and $V$, as well as part (a). To prove part (b), given the limit matrix $\call F \in \Lip(\Lambda;\Sym(2))$ of \Cref{rmk_lplf}, for any $\lambda \in \Lambda$, consider the system \eqref{NS} with $Q_{f^1 \otimes I_d}(\bs x) = Q_{\call F(\lambda)}(x^1,\beta(m_{\bs x^{-1}}))$ and $Q_{f^i \otimes I_d}(\bs x) = Q_{\bar f_N(\lambda^i_N) \otimes I_d}(x^i,\bs x^{-i})$ for $i \neq 1$; then, denoting by $\tilde u^T_N$ the solution to this system, by part (a) we have in particular, for any $x,b \in \R^d$,
\[
D_1^k \tilde u_N^{T,1}(t,(x,b,\dots,b)) \to D_x^k U^{T}(\lambda, \rho; t, x, b) \qquad \forall\,k \in\N.
\]
Therefore, by the above convergences, we can pass the Nash system to the large population limit and obtain \eqref{MEUT}, with $F = Q_{\call F}$ and $G = Q_{\call G}$. Similarly, one proves that $(V,\Gamma)$ solves \eqref{MEV}, with $\Gamma = \Delta_x V$. Finally, part (c) follows by letting $N \to \infty$ in \eqref{uiviconv}.
\end{proof}


\begin{thebibliography}{88}
\bibitem{BarPri} {\sc M.\ Bardi and F.\ S.\ Priuli}, Linear-quadratic $N$-person and mean-field games with ergodic cost, \emph{SIAM J.\ Control Optim.\ }{\bf52} (2014), 3022--3052.
\bibitem{BayJiaG} {\sc E.\ Bayraktar, Z.\ Cao, and J.\ Jian}, Long-time behavior and turnpike properties of linear-quadratic graphon mean field control problems, preprint arXiv:2607.18000 (2026).
\bibitem{BayJia25} {\sc E.\ Bayraktar and J.\ Jian}, Ergodicity and turnpike properties of linear-quadratic mean field control problems, preprint arXiv:2502.08935 (2025).
\bibitem{BayJia26} {\sc E.\ Bayraktar and J.\ Jian}, Uniform-in-time convergence and turnpike properties of linear-quadratic mean field control problems with common noise, preprint arXiv:2601.07815 (2026).
\bibitem{BenFre84} {\sc A.\ Bensoussan and J.\ Frehse}, Nonlinear elliptic systems in stochastic game theory, \emph{J.\ Reine Angew.\ Math.\ }{\bf350} (1984), 23--67.
\bibitem{BenFreProc} {\sc A.\ Bensoussan and J.\ Frehse}, Ergodic Bellman systems for stochastic games in arbitrary dimension, \emph{Proc.\ R.\ Soc.\ A} {\bf449} (1995), 65--77.
\bibitem{Car13} {\sc P.\ Cardaliaguet}, Long time average of first order mean field games and weak KAM theory, \emph{Dyn.\ Games Appl.\ }{\bf3} (2013), 473--488.
\bibitem{CarLC} {\sc P.\ Cardaliaguet}, The convergence problem in mean field games with local coupling, \emph{Appl.\ Math.\ Optim.\ }{\bf76} (2017), 177--215.
\bibitem{CCP} {\sc P.\ Cardaliaguet, M.\ Cirant, and A.\ Porretta}, Remarks on Nash equilibria in mean field game models with a major player, \emph{Proc.\ Amer.\ Math.\ Soc.\ }{\bf148} (2020), 4241--4255.
\bibitem{CDLL} {\sc P.\ Cardaliaguet, F.\ Delarue, J.-M.\ Lasry, and P.-L.\ Lions}, \emph{The master equation and the convergence problem in mean field games}, Annals of Mathematics Studies {\bf201}, Princeton University Press, Princeton,
NJ, 2019.
\bibitem{CG15} {\sc P.\ Cardaliaguet and P.\ J.\ Graber}, Mean field games systems of first order, \emph{ESAIM Control Optim.\ Calc.\ Var.\ }{\bf21} (2015), 690--722.
\bibitem{CLLP1} {\sc P.\ Cardaliaguet, J.-M.\ Lasry, P.-L.\ Lions, and A.\ Porretta}, Long time average of mean field games, \emph{Netw.\ Heterog.\ Media} {\bf7} (2012), 279--301.
\bibitem{CLLP2} {\sc P.\ Cardaliaguet, J.-M.\ Lasry, P.-L.\ Lions, and A.\ Porretta}, Long time average of mean field games with a nonlocal coupling, \emph{SIAM J.\ Control Optim.\ }{\bf51} (2013), 3558--3591.
\bibitem{CarRai} {\sc P.\ Cardaliaguet and C.\ Rainer}, An example of multiple mean field limits in ergodic differential games, \emph{Nonlinear Differ.\ Equ.\ Appl.\ }{\bf27} (2020), 2.
\bibitem{CCDE} {\sc A.\ Cecchin, G.\ Conforti, A.\ Durmus, K.\ Eichinger}, The exponential turnpike phenomenon for mean field game systems: weakly monotone drifts and small interactions, \emph{Electron.\ J.\ Probab.\ }{\bf31} (2026), 1--70.
\bibitem{CecDia} {\sc A.\ Cecchin and J.\ Dianetti}, Convergence for linear quadratic potential mean field games, preprint arXiv:2602.14842 (2026).
\bibitem{CJR25} {\sc M.\ Cirant, J.\ Jackson, and D.\ F.\ Redaelli}, A non-asymptotic approach to stochastic differential games with many players under semi-monotonicity, arXiv:2505.01526.
\bibitem{CM25} {\sc M.\ Cirant, A.\ R.\ M\'esz\'aros}, Long time behavior and stabilization for displacement monotone mean field games, preprint arXiv:2412.14903 (2025).
\bibitem{CP21} {\sc M.\ Cirant and A.\ Porretta}, Long time behavior and turnpike solutions in mildly non-monotone mean field games, \emph{ESAIM Control Optim.\ Calc.\ Var.\ }{\bf27} (2021), 86.
\bibitem{CR-DGA} {\sc M.\ Cirant and D.\ F.\ Redaelli}, Some remarks on linear-quadratic closed-loop games with many players, \emph{Dyn.\ Games Appl.\ }{\bf15} (2025), 558--591.
\bibitem{CR-CPAM26} {\sc M.\ Cirant and D.\ F.\ Redaelli}, A priori estimates and large population limits for some nonsymmetric Nash systems with semimonotonicity, \emph{Comm.\ Pure Appl.\ Math.\ }{\bf79} (2026), 3--88.
\bibitem{CohJia} {\sc A.\ Cohen and J.\ Jian}, Quantitative comparison of closed- and open-loop linear-quadratic N-player differential games, preprint arXiv:2609.06233 (2026).
\bibitem{CohZell} {\sc A.\ Cohen and E.\ Zell}, Asymptotic Nash equilibria of finite-state ergodic Markovian mean field games, \emph{Math.\ Oper.\ Res.\ }{\bf51} (2026), 1139--1173.
\bibitem{Dj} {\sc M.\ F.\ Djete}, Large population games with interactions through controls and common noise: convergence results and equivalence between open-loop and closed-loop controls, \emph{ESAIM Control Optim.\ Calc.\ Var.\ }{\bf29} (2023), 39.
\bibitem{GMMZ} {\sc W.\ Gangbo, A.\ R.\ M\'esz\'aros, C.\ Mou, and J.\ Zhang}, Mean field games master equations with nonseparable Hamiltonians and displacement monotonicity, \emph{Ann.\ Probab.\ }{\bf50} (2022), 2178--2217.
\bibitem{HuaZ} {\sc M.\ Huang and M.\ Zhou}, Linear quadratic mean field games: asymptotic solvability and relation to the fixed point approach, \emph{IEEE Trans.\ Autom.\ Control.\ }{\bf65} (2020), 1397--1412.
\bibitem{JM25} {\sc J.\ Jackson and A.\ R.\ M\'esz\'aros}, Quantitative convergence for displacement monotone Mean Field Games of control, arXiv:2507.17014.
\bibitem{JJT} {\sc J.\ Jackson and L.\ Tangpi}, Quantitative convergence for displacement monotone mean field games with controlled volatility, \emph{Math.\ Oper.\ Res.\ }{\bf49} (2023), 2527--2564.
\bibitem{LackAAP} {\sc D.\ Lacker}, On the convergence of closed-loop Nash equilibria to the mean field game limit, \emph{Ann.\ Appl.\ Probab.\ }{\bf 30} (2020), 1693--1761.
\bibitem{LackLF} {\sc D.\ Lacker and L.\ Le Flem}, Closed-loop convergence for mean field games with common noise, \emph{Ann.\ Appl.\ Probab.\ }{\bf 33} (2023), 2681--2733.
\bibitem{LMWZ} {\sc M.\ Li, C.\ Mou, Z.\ Wu, and C.\ Zhou}, Linear-quadratic mean field games of controls with non-monotone data, \emph{Trans.\ Am.\ Math.\ Soc.\ }{\bf376} (2023), 4105--4143.
\bibitem{LLL} {\sc P.-L. Lions}, Cours au Collège de France (2008--09), available at college-de-france.fr.
\bibitem{MZMem} {\sc C.\ Mou and J.\ Zhang}, Wellposedness of second order master equations for mean field games with nonsmooth data, \emph{Mem.\ Am.\ Math.\ Soc.\ }{\bf302}, 2024.
\bibitem{PriDGA} {\sc F.\ S.\ Priuli}, Linear-quadratic $N$-person and mean-field games: infinite horizon games with discounted cost and singular limits, \emph{Dyn.\ Games Appl.\ }{\bf5} (2015), 397--419.
\end{thebibliography}
\end{document}